\documentclass[10.5pt]{article}
\usepackage[affil-it]{authblk}
\usepackage{etex}
\usepackage{lineno,hyperref}

\usepackage{indentfirst}
\usepackage[below]{placeins} 

\usepackage{amsmath,amssymb,amsfonts,amsthm,latexsym}
\usepackage{amsfonts}
\usepackage[utf8]{inputenc} \usepackage{ae,aecompl}     
\usepackage{amscd}
\usepackage{mathrsfs}       
\usepackage{fancyhdr}
\usepackage{graphicx,psfrag}
\usepackage{epsf} 
\usepackage[all]{xy}
\usepackage{tikz}
\usetikzlibrary{chains,scopes,shapes,arrows,trees,matrix,positioning,fit,decorations.pathmorphing}

\usepackage[nofoot,letterpaper,left=1.75truein, top=1truein]{geometry}

\usepackage{multicol}

\theoremstyle{definition}

\newtheorem{thm}{Theorem}[section]

\newtheorem{lem}[thm]{Lemma}

\newtheorem{cor}[thm]{Corollary}
\newtheorem{pro}[thm]{Proposition}

\theoremstyle{definition}

\newcommand {\C}{\mathbb C}
\newcommand {\Z}{\mathbb Z}
\newcommand {\R}{\mathbb R}

\newcommand {\nb}{\mathcal N}

\begin{document}

\title{	$PSL_2(\C)$-character varieties and Seifert fibered cosmetic surgeries}
\author{Huygens C. Ravelomanana}
\date{}

\maketitle

\begin{abstract}
We study small Seifert possibly chiral cosmetic surgeries on not necessarily null-homologous knot in rational homology spheres. Using $PSL_2(\C)$-character variety theory we give a sharp bound on the number of slopes producing the same small Seifert manifold if the ambient manifold satisfies some representation theoretic conditions. 
\end{abstract}


\section{Introduction}

Let $Y$ be a rational homology sphere and $K$ be a knot in $Y$ such that $Y_K:= Y\setminus int(\nb(K)) $ is \emph{boundary irreducible} and \emph{irreducible}.  Two Dehn surgeries $Y_K(r)$ and $Y_K(s)$ with distinct slopes are called ``\emph{cosmetic}" if they are homeomorphic. They are called ``\emph{truly cosmetic}" if the homeomorphism preserve orientation and ``\emph{chirally cosmetic}" if the homeomorphism reverse orientation. It is conjectured that truly cosmetic surgery on such knot $K$ does not exist \cite{Kirby-list}. In this paper we will focus on ``general" cosmetic surgery and will not distinguish  between chiral and truly cosmetic surgeries. 

Let's fix a slope $s$ and define
$$C(s)=\left\lbrace \textnormal{slope}\ \  r \neq s | \ Y_K(r) \cong Y_K(s)  \right\rbrace.$$ 
Here we allow the homeomorphism to reverse the orientation. If  $C(s) \neq \emptyset$ then we have a cosmetic surgery (possibly chiral). The following theorem then gives a bound on the number of element in $C(s)$.

\begin{thm}\label{my-theorem-I}
Let $K$  be a small knot in $Y$ and $Y_K(s)$ be small-Seifert.  If $\text{Hom}\left(\pi_1(Y),PSL_2(\C)\right)$ contains only diagonalisable representations and $||s||$ is not a multiple of $s \cdot \lambda$. Then $\sharp C(s)\leq 1$. 
\end{thm}

Here , $||\hspace{0.2cm} ||$ is a semi-norm on $H_1(\partial Y_K;\R)$ similar to the Culler-Shalen semi-norm,  $\lambda$ is the rational longitude of $K$ and $s \cdot \lambda$ is the algebraic intersection. The knot $K$ being small means that its exterior does not contain closed ``\emph{incompressible surfaces}".

The bound in the theorem is sharp since amphicheiral knot in $S^3$ admits ``chiral" cosmetic surgeries. Moreover in \cite{Weeks}, we can find a construction of a one-cusped hyperbolic 3-manifold with a pair of distinct slopes which gives oppositely oriented copies of the lens space $L(49/18)$. An infinite family of hyperbolic manifolds which admit pairs $\left\lbrace\alpha,\beta\right\rbrace$ of reducible filling slopes, of which some pairs yield homeomorphic manifolds are presented in \cite{Matignon}.  A straightforward consequence of the theorem which relates to the cosmetic surgery conjecture is the following.

\begin{cor}
Under the hypothesis of the theorem, there are at most two distinct slopes which can produce the same oriented manifold after surgery.
\end{cor}
Result similar to this corollary has been proven for null homologous knot in 3-manifold with positive first Betti number \cite{Ni-cosmetic-Thurston-norm}, for knot in $S^3$ \cite{RavelomananaS3-paper, Ni-Wu} and for knot in L-space integer homology sphere, \cite{Wu}.

The main ingredient for the proof of Theorem~\ref{my-theorem-I} is the theory of 3-manifold character variety started by Culler and Shalen \cite{Culler-Shalen} and which was essential in the proof of the ``\emph{cyclic surgery theorem}" \cite{Cyclic-paper}, the ``\emph{finite surgery theorem}" \cite{finite-paper} and is also an useful tool for studing topological properties of knot exterior. We will use results from \cite{Leila-2} together with a norm derived from $PSL_2(\C)$-character variety similar to the ``Culler Shalen norm" .  

\paragraph*{Organization.} In section \ref{section:charac-preliminaries} we give some background on character varieties and the Culler-Shalen norm. The proof of Theorem~\ref{my-theorem-I} will be given in section \ref{small-seifert}.

\paragraph*{Acknowledgment.}
The author is  grateful to Steven Boyer for advices, comments and helpful discussions. The author also thanks the University of Georgia for its support, the CIRGET group at UQ\`AM in Montr\'eal.

\section{Character varieties}\label{section:charac-preliminaries}
Let $M$ denote a compact orientable 3-manifold which is \emph{irreducible} and \emph{boundary irreducible} with boundary consisting of a single torus. Typically $M$ would be the knot exterior $Y_K$ in the introduction. We recall that \emph{irreducible} means every embedded 2-sphere bounds a 3-ball and \emph{boundary irreducible} means every simple closed curve on $\partial M$ which bounds a disk in $M$ bounds a disk in $\partial M$. A properly embedded surface in $M$ will be called \emph{essential} if it is not boundary parallel, it is not a 2-sphere and is $\pi_1$-injective. We say that $M$ is \emph{small} if it does not contain closed essential surfaces.  A slope is the isotopy class of a simple closed curve on $\partial M$. Since $M$ is boundary irreducible, $\pi_1(\partial M)\hookrightarrow \pi_1(M)$ therefore using the fact that $H_1(M)\cong \pi_1(\partial M)$ we will think of a slope as both being an element of $H_1(\partial M),\; \pi_1(\partial M)$ or $\pi_1(M)$. A slope will be called ``\emph{boundary slope}"  if it corresponds to the boundary of an essential surface.  It will be called \textit{strict boundary slope}" if it corresponds to the boundary of an essential surface which is not a (semi) fibre in any (semi) fibration of $M$ over $S^1$. 

The \textit{$PSL_2(\C)$-representation variety} of $M$ is the set $\overline{R}(M):= \hom (\pi_1(M), PSL_2(\C))$ equipped with the compact-open topology. It consists of representations of $\pi_1(M)$ to $PSL_2(\C)$. The space $\overline{R}(M)$ has the structure of an affine complex algebraic set \cite{Culler-Shalen}. The group $PSL_2(\C)$ acts algebraically on  $\overline{R}(M)$ by conjugation. Two representations are called equivalent if they are conjugate to each other. If two representations are equivalent then they belong to the same irreducible component of $\overline{R}(M)$ \cite{Culler-Shalen}. The set of equivalence classes of representations corresponds to the quotient  $\overline{R}(M)// PSL_2(\C)$, where the quotient is taken in the algebraic geometric category. In order to understand this set, Culler and Shalen introduced  the $PSL_2(\C)$-\textit{character variety} of $M$ using the \textit{trace function}. For each representation $\rho \in \overline{R}(M)$, the $PSL_2(\C)$-\textit{character} of $\rho$ is the map   $\chi_{\rho}$ defined by 
 $$\chi_{\rho}: \pi_1(M) \to \C, \ \ \chi_{\rho(g)}=\text{trace}\left(\rho(g)\right)^2.$$
 The set of all characters $\overline{X}(M)=\left\lbrace \chi_{\rho}\ |\ \rho \in \overline{R}(M) \right\rbrace$ is also a complex algebraic set in a natural way such that the following map is regular, in the sense of algebraic geometry,
$$\overline{t}: \overline{R}(M) \longrightarrow \overline{X}(M),\ \   \overline{t}(\rho) = \chi_{\rho}.$$
Moreover its corresponds to the  quotient  $\overline{R}(M)// PSL_2(\C)$ \cite{Culler-Shalen}. Let $\overline{R}^{irr}(M)$  be the subset of irreducible representations and let  $\overline{X}^{irr}(M)=\overline{t}\left(\overline{R}^{irr}(M)\right)$. The spaces $\overline{R}^{irr}(M)$ and $\overline{X}^{irr}(M)$ are Zariski open subsets of $\overline{R}(M)$ and $\overline{X}(M)$ respectively \cite{Culler-Shalen}.

For each $\gamma \in \pi_1(M)$ we consider the function defined by 
$$ f_{\gamma}: \overline{X}(M) \longrightarrow \C,\ \   f_{\gamma}(\chi ) =  \textnormal{trace}(\rho(\gamma))^2-4 =\chi(\gamma)-4.$$
The function  $f_{\gamma}$ is a regular function and the zeros of $f_{\gamma}$ are the characters of representations $\overline{\rho}$ for which $\overline{\rho}(\gamma)$ is parabolic or $\overline{\rho}(\gamma) = \left[\pm \textnormal{Id}\right]$. We will use the same notation $f_{\gamma}$ for the restriction of $f_{\gamma}$ to a curve $X_0\subset \overline{X}(M)$.
  
Let $X_0 \subset \overline{X}(M)$ be a {\em non-trivial} irreducible curve component. Here {\em non-trivial} means that it contains the character of an irreducible representation. Let $\widehat{X}_0$ be the normalized projective completion of $X_0$. There is an isomorphism between function fields 

$$\C(X_0) \cong \C(\widehat{X}_0), \ \ \ f\longmapsto \widehat{f}.$$
 We can then define the degree of $f$ to be the degree of $\widehat{f}$. For $x\in \widehat{X}$ we denote by  $Z_x(\widehat{f}_{\gamma})$ the multiplicity of $x$ as a zero of $\widehat{f}_{\gamma}$. By convention $Z_x(\widehat{f}_{\gamma})=\infty$ if $\widehat{f}_{\gamma} \equiv 0$. Now we denote $\Lambda = \pi_1(\partial M)$ seen as a subgroup of $\pi_1(M)$. We can also think of $\Lambda$ as a lattice in $H_1(\partial M, \R)$. An element $\gamma \in \Lambda$ satisfies \cite{Shalen},
 $$\textnormal{deg} (\widehat{f}_{\gamma}) = \sum_{x\in \widehat{X}_0}\  Z_x(\widehat{f}_{\gamma}).$$ 
The degree is finite if $\widehat{f}_{\gamma}$ is non-constant on $\widehat{X}_0$. The  key property of $\textnormal{deg} (\widehat{f}_{\gamma})$ is that for each curve
 $X_0 \subset X(M)$ it defines a semi-norm $||.||_{X_0}$ on  $H_1(\partial M, \R)$ which for each $\gamma \in \Lambda$ satisfies  \cite{Boyer-Zhang, Shalen}
$$||\gamma||_{X_0} =\left\lbrace \begin{aligned}&0\ \ \textnormal{if }\ \  f_{\gamma}|X_0 \textnormal{ is constant} \\
 &\textnormal{deg} (\widehat{f}_{\gamma}) \ \ \textnormal{if }\ \  f_{\gamma}|X_0\neq 0.\end{aligned}\right.$$
 This semi-norm is called the \textit{Culler-Shalen semi-norm} associated to the curve $X_0$.

Note that if $B_k$ is the ball of radius $k$ centred at the origin, then $B_k$ can be viewed as the unit ball for the norm $\frac{1}{k}||.||_{X_0}$ therefore $B_k$ has the same properties as the unit ball. Some of these properties are \cite[Proposition 5.2, Proposition 5.3]{Boyer-Zhang}: if it is a norm then the unit ball is a balanced convex polygon and if it is not a norm but is a non-zero semi-norm them the unit ball is an infinite band.

Let $X_1,\cdots, X_k$ be all the non-trivial irreducible curve components in $\overline{X}(M)$. We can define an ``absolute" semi-norm $||.||$ on  $H_1(\partial M, \R)$ by  
$$||.||=||.||_{X_1}+\cdots+ ||.||_{X_k}.$$
We will call this semi-norm the {\em absolute semi-norm}.
\bigskip

 There is a unique 4-dimensional subvariety $R_0 \subset \overline{R}(M)$ for which $t(R_0)=X_0$, see \cite[Lemma 4.1]{Boyer-Zhang}. 
If $\overline{\rho}(\alpha)= \pm I$ for some slope $\alpha$, then we have an induced representation $\overline{\rho}':\pi_1(M(\alpha))\longrightarrow PSL_2(\C)$ and a cohomology group $H^1(M(\alpha);Ad_{\overline{\rho}'})$ .

\bigskip

Let  $\nu : X_0^{\nu}:=\widehat{X}_0\setminus \left\lbrace \textnormal{ideal points} \right\rbrace \longrightarrow X_0$ be the map which corresponds to the affine normalization of $X_0$.
There is an affine normalization $R_0^{\nu}\longrightarrow R_0$, which we still denote by $\nu$, such that the following diagram commutes.  
\begin{center}
\begin{tikzpicture}[node distance=2cm,auto]
  \node (P) {$R_0^{\nu}$};
  \node (B) [right of=P] {$R_0$};
  \node (A) [below of=P] {$X_0^{\nu}$};
  \node (C) [below of=B] {$  X_0 $};
  \draw[->] (P) to node {$\nu$} (B);
  \draw[->] (P) to node [swap] {$\overline{t}^{\nu}$} (A);
  \draw[->] (A) to node [swap] {$\nu$} (C);
  \draw[->] (B) to node {$\overline{t}$} (C);
\end{tikzpicture}
\end{center}
The map $t^{\nu}$ and $\nu$ are all surjective, see \cite{Cyclic-paper}.  

Let $N \subset PSL_2(\C)$ denote the subgroup
$$\left\lbrace \pm \left( \begin{array}{cc}
z & 0  \\
0 & z^{-1} \\
\end{array} \right)  
, \pm \left( \begin{array}{cc}
0 & w  \\
-w^{-1} & 0 \\
\end{array} \right)  
\ \ | \ \ z,w\in \C^*  \right\rbrace$$

For each $\gamma \in \pi_1(M)$ we are going to consider the following subset of $\overline{X}(M)$:
$$A(\gamma)=\left\lbrace \chi_{\overline{\rho}} \in  \overline{X}(M)\ |\ \overline{\rho}(\gamma)=\pm I; \ \ \ \ \overline{\rho}  \ \ \textnormal{ is non-abelian and conjugates into}\ N   \right\rbrace.$$ 
$$
\begin{aligned}
B(\gamma)=\left\lbrace\right. &\chi_{\overline{\rho}} \in  \overline{X}(M)\ |\ \overline{\rho}(\gamma)=\pm I;  \ \ \ \ \overline{\rho} \ \ \textnormal{is non-abelian and does not} \\
& \textnormal{ conjugates into}\ N   
\left. \right\rbrace.
\end{aligned}$$ 
Note that elements of $A(\gamma)$ must be irreducible but not necessarily those of $B(\gamma)$.

The following Theorem and proposition relate $Z_x(\widehat{f}_{\alpha})$ and $Z_x(\widehat{f}_{r})$  for two slopes $\alpha$ and $r$ when $x \in \widehat{X}_0$  is a regular point.


\begin{thm}\label{leila-theorem}\cite{Leila-2}
Fix a slope $\alpha$ on $\partial M$ and consider a non-trivial, irreducible curve $X_0 \subset \overline{X}(M)$. Suppose that $x \in \widehat{X}_0$ is not an ideal point and corresponds to a character $\chi_{\rho}$ for some representation $\overline{\rho} \in R_0$ with non-abelian image and which satisfies $\rho(\alpha)\in \left\lbrace \pm I \right\rbrace$.
 Assume that $H^1(M(\alpha);Ad_{\rho})=0$ and $\rho(\pi_1(\partial M)) \nsubseteq \left\lbrace \pm I \right\rbrace$.
 \begin{enumerate}
\item If $\beta \in \pi_1(\partial M)$ and $\rho(\beta) \neq \pm I$, then  
$$Z_x(\widehat{f}_{\alpha})\geq  \left\lbrace \begin{aligned}& Z_x(\widehat{f}_{\beta})+1\ \ \textnormal{if }\ \overline{\rho} \ \text{conjugates into } \ N, \\
 &  Z_x(\widehat{f}_{\beta})+2\  \  \text{otherwise}.\end{aligned}\right.
$$
\item If $r \in \pi_1(\partial M)$ and $Z_x(\widehat{f}_{\alpha})>Z_x(\widehat{f}_{\beta})$, then 
$\widehat{f}_{\alpha}|X_0 \neq 0,\ \rho(r)\neq \pm I $ and
$$ Z_x(\widehat{f}_{\alpha})= \left\lbrace \begin{aligned}& Z_x(\widehat{f}_{\beta})+1\ \ \textnormal{if }\ \overline{\rho} \ \text{conjugates into } \ N, \\
 &  Z_x(\widehat{f}_{\beta})+2\  \  \text{otherwise}.\end{aligned}\right.
$$
 
 \end{enumerate}
\end{thm}

\bigskip

The condition  $\rho(\pi_1(\partial M)) \nsubseteq \left\lbrace \pm I \right\rbrace$ may not be satisfied in general. For it to be true we will assume the auxiliary assumption that the manifold $Y$ has only diagonalisable $PSL_2(\C)$ representations. 


\begin{pro}\cite{Cyclic-paper}\label{inequality}
Let $\alpha$ and $\beta$ be non-zero elements of $\Lambda$. Suppose that $x$ is a point of $X^{\nu}_0$ such that $Z_x(f_{\alpha}) > Z_x(f_{\beta})$. Then for every $\tilde{\rho} \in R_0^{\nu}$ with $t^{\nu}(\tilde{\rho})=x$, the representation $\rho=\nu(\tilde{\rho})$ satisfies $\rho(\alpha)=\pm I$.
\end{pro}

When we have zeros at ideal points we have the following property.

\begin{pro}\cite{Cyclic-paper}\label{ideal-point-inequality}
Let $x$ be an ideal point of $\widehat{X}_0$. Let $\alpha$ and $\beta$ be non-zero elements of $\Lambda$. Suppose that $\alpha$ is primitive and is not a boundary class,  and that
$$Z_x(f_{\alpha}) > Z_x(f_{\beta}).$$
Then there is a closed essential surface in $M$ which is incompressible in $M(\alpha)$.
\end{pro}



 \section{Proof of theorem \ref{my-theorem-I}}\label{small-seifert}

From now on we consider the case $M=Y_K$ the knot exterior described in the introduction.
 \begin{lem}\label{existence-of-nonabelian}
 Assume that $\textnormal{rank}_{\Z}\left(H_1(Y_K)\right)=1$. Then for each ordinary point $x\in \widehat{X}_0$ there is a representation $\overline{\rho} \in R_0$, with non-abelian image, such that $\chi_{\overline{\rho}}=\nu(x)$.
 \end{lem}
\begin{proof}

  Let $Z_0\subset X(Y_K)$ ($SL_2(\C)$-character variety) be an irreducible curve component of $\pi^{-1}(X_0)$, and $S_0$ a component of $t^{-1}(Z_0)$.

  Let $X(\Gamma)$ be the $SL_2(\C)$-character variety of a finitely generated group $\Gamma$. In \cite[Proposition 2.8]{Boyer}  it is shown that if $x$ is a reducible trivial character in a non-trivial curve inside $X(\Gamma)$ then the first Betti number satisfies $b_1(\Gamma) \geq 2$.  Since   $\textnormal{rank}_{\Z}\left(H_1(Y_K)\right)=1$ by assumption and $b_1\left(\pi_1(Y_K)\right)=\textnormal{rank}_{\Z}\left(H_1(Y_K)\right)$, any character in a non-trivial curve inside $X(Y_K)$ is non-trivial, in particular any element of $Z_0$ is non-trivial. The same Proposition 2.8 of \cite{Boyer} applied to $\pi_1(Y_K)$ implies that if a character $z\in Z_0$ is non-trivial then there is a representation $\rho \in S_0 \cap t^{-1}(z)$ with non-abelian image.  Since for each $x\in \widehat{X}_0$,  $\nu(x)\in X_0$ we can take  $z\in \pi^{-1}(\nu(x))$ to get such a representation $\rho$ and then take the corresponding $PSL_2(\C)$ representation $\overline{\rho}$.
\end{proof} 
 
 
\begin{lem}\label{non-zero-norm} If $Y_K$ is small then for  each  curve $X_0 \subset X(Y_K)$,  $||.||_{X_0}$ is not identically zero. 
\end{lem}
 \begin{proof} By \cite[Proposition 5.5]{Boyer-Zhang} if $||.||_{X_0} \equiv 0$ then $Y_K$ contains a closed essential surface, this is not possible if $Y_K$ is small. 
 \end{proof}

\begin{lem}\label{correspondence}  If $s$ and $r$ be two slopes on $\partial Y_K$ such that $\pi_1(Y_K(s))\cong \pi_1( Y_K(r))$. Then there is a one to one correspondence between $A(s)$ and $A(r)$, and  between $B(s)$ and $B(r)$,
\end{lem} 
 \begin{proof}
 
  Let $\Psi: \pi_1(Y_K(r)) \to \pi_1( Y_K(s))$ be an isomorphism, $p_{s}: \pi_1(Y_K) \to \pi_1(Y_K(s))$, and $ \ p_{r}:\pi_1(Y_K) \to \pi_1(Y_K(r))$ be the obvious projections. Let $\chi_{\overline{\rho}}\in A(s)$, we have a representation  $\Phi_{s} (\overline{\rho}): \pi_1(Y_K(s)) \to PSL_2(\C)$   obtained via the following factorisation  of $\overline{\rho}$

\begin{center}
\begin{tikzpicture}[node distance=2cm,auto]
  \node (G) {$\pi_1(Y_K)$};
  \node (H2) [below of=G] {$\pi_1(Y_K(s))$};
  \node (PSL) [right of=H2,node distance=3cm] {$ PSL_2(\C)$};
  \draw[->] (G) to node [swap] {$p_{s}$} (H2);
  \draw[->,dashed] (H2) to node [swap] {$\Phi_{s}(\overline{\rho})$} (PSL);
  \draw[->] (G) to node [ ] {$\overline{\rho}$} (PSL);
\end{tikzpicture}
\end{center}
We also have an equivalent representation  $\Phi_{r} (\overline{\rho}): Y_K(r) \to PSL_2(\C)$   for the $r$-case. 
Let $\overline{\rho}'$ be the composition  $\overline{\rho}':= \Phi (\overline{\rho}) \circ \Psi \circ p_{r}$
 \begin{center}
\begin{tikzpicture}[node distance=2cm,auto]
  \node (G) {$\pi_1(Y_K)$};
  \node (H1) [below of=G] {$\pi_1(Y_K(r))$};
  \node (H2) [right of=H1, node distance=3cm] {$\pi_1(Y_K(s))$};
  \node (PSL) [right of=G,node distance=3cm] {$ PSL_2(\C)$};
  \draw[->] (G) to node [swap] {$p_{r}$} (H1);
  \draw[->] (H1) to node [swap] {$\Psi$} (H2);
  \draw[->] (H2) to node [swap] {$\Phi(\overline{\rho})$} (PSL);
  \draw[->, dashed] (G) to node  {$\overline{\rho}'$} (PSL);
\end{tikzpicture}
\end{center}
 The maps $p_{r}, p_{s}$ and $\Psi$ are all surjective so $\textnormal{im} \overline{\rho} = \textnormal{im} \overline{\rho}'$. In particular if $\overline{\rho}$ does not conjugate into $N$ then neither does $\overline{\rho}'$, and if $\overline{\rho}$ is irreducible then so is $\overline{\rho}'$.
The representation $\overline{\rho}'$ satisfies $\overline{\rho}'(r)=\pm I$ by construction. Next we need to check that if $\chi_{\overline{\rho_1}}=\chi_{\overline{\rho}_2}$ then $\chi_{\overline{\rho}_1'}=\chi_{\overline{\rho}_2'}$.

We first assume that  $\chi_{\overline{\rho}_1}$ is an irreducible character. Therefore $\overline{\rho}_2=\overline{g}\ \overline{\rho}_1\ \overline{g}^{-1}$ for some \\ $g\in SL_2(\C)$. Then we deduce that
$$\begin{aligned}
\overline{\rho}_2' &= \Phi_{s}\left(\overline{\rho}_2\right)\circ \Psi \circ p_{r}  \\
  & = \Phi_{s}\left(\overline{g}\ \overline{\rho}_1 \ \overline{g}^{-1}\right)\circ \Psi \circ p_{r}\\ 
  & =  \left(\overline{g}\ \Phi_{s}(\overline{\rho}_1) \ \overline{g}^{-1}\right)\circ \Psi \circ p_{r} \\
  & = \overline{g}\ \left(\Phi_{s}(\overline{\rho}_1) \circ \Psi \circ p_{r} \right)\  \overline{g}^{-1} \\
  & = \overline{g} \ \overline{\rho}'_1  \overline{g}^{-1} 
\end{aligned}
$$
which implies  $\overline{\rho}_2'=\overline{g}\ \overline{\rho}_1'\ \overline{g}^{-1}$. Therefore $\chi_{\overline{\rho}_1'}=\chi_{\overline{\rho}_2'}$. If $\gamma \in \pi_1(\partial Y_K)$ we denote $X^{irr}(\gamma)$ and $X^{red}(\gamma)$  the sets 
$$X^{irr}(\gamma)=\left\lbrace \chi_{\overline{\rho}} \in  \overline{X}(Y_K)\ |\ \overline{\rho}(\gamma)=\pm I, \ \ \text{and} \ \overline{\rho}  \ \textnormal{ is irreducible}\right\rbrace$$ 
$$X^{red}(\gamma)=\left\lbrace \chi_{\overline{\rho}} \in  \overline{X}(Y_K)\ |\ \overline{\rho}(\gamma)=\pm I, \ \ \text{and} \ \overline{\rho}  \ \textnormal{ is reducible}\right\rbrace.$$

 We then have a well defined map
$$F: X^{irr}(s) \longrightarrow X^{irr}(r),\ \ \overline{\rho} \longmapsto \overline{\rho}'.$$
The map $F$ sends $A(s)$ to $A(r)$, and  $B(s)$ to $B(r)$. The next step is to extend $F$ to the reducible representations. 

 Now assume $x=\chi_{\overline{\rho}_1}=\chi_{\overline{\rho}_2}$ is a reducible character. By analogy with \cite{Boyer} there exists a representation $a: \pi_1(Y_K)\to \C^*$ such $t^{-1}(x)=R^a_x \cup R^{a^{-1}}_x$ where
 $$R^a_x=\left\{ A\overline{\rho} A^{-1}| A\in PSL_2(\C),\ \overline{\rho} \in U^a_x\right\},$$
 $$ R^{a^{-1}}_x=\left\{ A\overline{\rho} A^{-1}| A\in PSL_2(\C),\ \overline{\rho} \in U^{a^{-1}}_x\right\}$$
 $$U^a_x=\left\{\textnormal{representation }\  \overline{\rho} = \pm
\left( \begin{array}{cc}
a & b  \\
0 & a^{-1} \\
\end{array} \right)  
  \right\},$$
  $$U^{a^{-1}}_x=\left\{\textnormal{representation }\  \overline{\rho} = \pm
\left( \begin{array}{cc}
a^{-1} & b  \\
0 & a \\
\end{array} \right)  
  \right\}$$
 
If $\overline{\rho}_1$ and $\overline{\rho}_2$ are conjugate, by the same argument as for irreducible characters $\chi_{\overline{\rho}_1'}=\chi_{\overline{\rho}_2'}$. Assume that  $\overline{\rho}_1$ and $\overline{\rho}_2$ are not conjugate. Without loss of generality we can suppose that 

$$\overline{\rho}_1=\pm \left( \begin{array}{cc}
a & b  \\
0 & a^{-1} \\
\end{array} \right)  
  \ \ \ \textnormal{and }\  \overline{\rho}_2 = \pm A
\left( \begin{array}{cc}
a^{-1} & b  \\
0 & a \\
\end{array} \right)  A^{-1},\ \ \  \text{for some} \ \ \ A \in PSL_2(\C).
$$

Since $\Phi(\overline{\rho}_i)$ and  $\overline{\rho}_i,\  \ i=1,2$ have the same image we can use the same matrices to represent $\Phi(\overline{\rho}_i), \ \ i=1,2$. Hence

$$\overline{\rho}'_1=\Phi (\overline{\rho}_1) \circ \Psi \circ p_{r}=\pm \left( \begin{array}{cc}
a\circ \Psi \circ p_{r} & b\circ \Psi \circ p_{r}  \\
0 & a^{-1}\circ \Psi \circ p_{r} \\
\end{array} \right)  $$

$$\textnormal{and }\  \overline{\rho}'_2=\Phi (\overline{\rho}_2) \circ \Psi \circ p_{r} = \pm A
\left( \begin{array}{cc}
a^{-1}\circ \Psi \circ p_{r}  & b\circ \Psi \circ p_{r}  \\
0 & a\circ \Psi \circ p_{r} \\
\end{array} \right)  A^{-1}.
$$

Therefore
$$\textnormal{trace}(\rho'_1)=\pm \left[ a\circ \Psi \circ p_{r}  +a^{-1}\circ \Psi \circ p_{r} \right]$$
$$\textnormal{trace}(\rho'_2)=\pm \left[ a^{-1}\circ \Psi \circ p_{r}  +a\circ \Psi \circ p_{r} \right]$$
It follows that $\textnormal{trace}(\rho'_1)^2=\textnormal{trace}(\rho'_2)^2$, that is $\chi_{\rho_1'}=\chi_{\rho_2'}$.
Thus $F$ is well defined on the set of reducible characters. 

Finally, we show that $F$ is bijective.  If $\overline{\rho}'\in X^{irr}(r)$ (resp. $X^{red}(r)$) , we get  $\overline{\rho}\in X^{irr}(s)$ (resp. $X^{red}(s)$) as follow: we first define $\Phi_{s}(\overline{\rho})$ to be  $\Phi_{s}(\overline{\rho})=\Phi_{r}(\overline{\rho}')\circ \Psi^{-1}$ then $\overline {\rho}'=\Phi_{s}(\overline{\rho})\circ p_{s}$. This uniquely determine $\overline{\rho}$, therefore the map $F$ is bijective. 
\end{proof}


\begin{lem}\label{mitovy-norm}
 Let $s$ be a slope on $\partial Y_K$ which is not a boundary slope. Assume that $Y_K(s)$ is small-Seifert, $Y_K$ is small, $Y$ has only diagonalisable $PSL_2(\C)$-representations and $b_1(Y_K(s))=0$. If  $r$ is a slope such that $Y_K(r) \cong Y_K(s)$, then $||r||=||s||$.
\end{lem}
\begin{proof}

It is known from \cite{plane-curve} that $X(Y_K)$ has no 0-dimensional component. Let $X_1,\cdots, X_k$ be the curve components of $X(Y_K)$. Since $Y_K$ is small $X(Y_K)$ is the union of these components. If $s \in \Lambda$ is not a boundary slope, then Lemma \ref{non-zero-norm} allows us to write the $X_i$-norm of $s$ in terms of the zeros of $f_{s}|_{X_i}$ for each $i\in \left\{1,\cdots,k\right\}$

$$||s||_{X_i} = \sum_{x\in \widehat{X_i}}\  Z_x(\widehat{f}_{s}|_{\widehat{X_i}}).$$

Let $ \widehat{X}$ be the abstract disjoint union of all the  $\widehat{X_i}$, $i\in \left\{1,\cdots,k\right\}$, then we have the following formula for the absolute semi-norm 
$$||s||= \sum_{x\in \widehat{X}}\  Z_x(\widehat{f}_{s})$$
where $\widehat{f}_{s}$ is understood to be the restriction to the appropriate component. Let $x\in \widehat{X}$, we define  the number $m_x$ and $m_0$ to be
$$m_x = \min \left\lbrace Z_x(\widehat{f}_{\gamma})\ |\  \gamma \in \Lambda\setminus\left\lbrace 0\right\rbrace \right\rbrace,  \ \ \ \textnormal{and} \ \ \ m_0= \sum_{x\in \widehat{X}}\  m_x.$$
We can then deduce 
$$
||s|| = m_0-m_0 + \sum_{x\in \widehat{X}}\  Z_x(\widehat{f}_{s}) 
= m_0+   \sum_{x\in \widehat{X}}\  \left(Z_x(\widehat{f}_{s})-m_x \right).
$$
\bigskip

Let us suppose that $x$ is an ideal point of $\widehat{X}_0\subset \widehat{X}$. If $Z_x(\widehat{f}_{s})-m_x >0 $ then $Z_x(\widehat{f}_{s}) >Z_x(\widehat{f}_{\gamma})$ for some $\gamma \in \Lambda \setminus \left\{0\right\}$. Since $s$ is primitive and is not a  boundary class, Lemma \ref{ideal-point-inequality} implies that there is a closed surface in $Y_K$ which is incompressible in $Y_K(s)$. This situation does not occur if we assume $Y_K(s)$ is small-Seifert with $b_1(M(s))=0$. Therefore we always have $Z_x(\widehat{f}_{s})-m_x = 0 $ at an ideal point.

Let $x\in \widehat{X}$ be an ordinary point, $x$ is contained in some $\widehat{X}_0$ and by Lemma \ref{existence-of-nonabelian} there is a representation $\overline{\rho} \in R_0$  with non-abelian image, such that $\chi_{\overline{\rho}}=\nu(x)$. Let $\tilde{\rho}=\nu^{-1}(\rho)$, we have the following equality
$$\nu\left(t^{\nu}(\tilde{\rho})\right)=t\left(\rho\right) =\nu(x).$$
The normalization map $\nu: X_0^{\nu} \to X_0$ is an ``isomorphism" outside singular points, so if $x$ is a smooth point then $t^{\nu}(\tilde{\rho})=x$. This smoothness is provided by Theorem A of \cite{Boyer}. A direct consequence of this is that for an ordinary point $x$, $\nu(x)$ is contained in only one irreducible component. Therefore if we consider instead of $\widehat{X}$, the ``natural" union $\widehat{X_1}\cup \cdots \cup \widehat{X_k}$, we can write the absolute semi-norm of $s$ as 

$$||s||= \sum_{x\in \widehat{X_1}\cup \cdots \cup \widehat{X_k}}\  Z_x(\widehat{f}_{s}).$$

 Now if we assume that $Z_x(\widehat{f}_{s}) > m_x$ then by Lemma \ref{inequality}  the representation $\rho=\nu\left(\tilde{\rho}\right)$  satisfies $\rho(s)=\pm I$. Since $Y_K(s)$ is Small seifert, $H^1(Y_K(s);Ad_{\rho})=0$. If we add the extra condition that $Y$ have only Abelian  $PSL_2(\C)$-representations then $\rho(\Lambda) \nsubseteq \left\{\pm I\right\}$ and all the hypothesis of Theorem \ref{leila-theorem} are satisfied. In particular if  $m_x = Z_x(\widehat{f}_{r})$ for some $r \in \Lambda \setminus \left\{0\right\}$ then 
$\widehat{f}_{s}|X_0 \neq 0,\ \rho(r)\neq \pm I $ and
$$ \left\lbrace \begin{aligned}& Z_x(\widehat{f}_{s})-m_x=Z_x(\widehat{f}_{s})-Z_x(\widehat{f}_{r})=1\ \ \textnormal{if }\ \overline{\rho} \ \text{conjugates into } \ D, \\
 &   Z_x(\widehat{f}_{s})-m_x=Z_x(\widehat{f}_{s})-Z_x(\widehat{f}_{r})=2\  \  \text{otherwise}.\end{aligned}\right.
$$
Since $X(Y_K)$ is the union of its 1-dimensional components we have
$$||s||=  m_0+   \sum_{x\in \widehat{X}}\  \left(Z_x(\widehat{f}_{s})-m_x \right) = m_0+A(s)+2B(s).$$


Finally $\sharp A(s)=\sharp A(r)$ and $\sharp B(s)=\sharp B(r)$ from Lemma~\ref{correspondence}, since $Y_K(s)\cong Y_K(r)$. Therefore $||r||=||s||$.

\end{proof}

\bigskip

\textbf{Proof of Theorem \ref{my-theorem-I}.}  Let $s$ be an exceptional slope on $\partial Y_K$ and $r\in C(s)$.


We can assume that $b_1(Y_K(s))= 0$ since there is only one slope which produces a manifold with positive first Betti number. Let us suppose that $s$ is a boundary slope. From \cite[Theorem 2.0.3]{Cyclic-paper} we have the following possibilities:
\begin{itemize}
\item[(i)] $Y_K(s)$ contains a closed essential surface of strictly positive genus.
\item[(ii)] $Y_K(s)$ is the connected sum of two lens spaces.
\item[(iii)] There is a closed essential surface $S \subset Y_K$ which compresses in $Y_K(s)$ but which remains incompressible in $Y_K(\delta)$ as long as $\Delta(s,\delta)> 1$.
\item[(iv)] $Y_K(s) \cong S^1 \times S^2$.
\end{itemize}

Since $Y_K(s)$ is small-Seifert with $b_1(Y_K(s))= 0$, only (iii)  can occur. Then the fact that $Y_K(s) \cong Y_K(r)$ implies that $S$ also compresses in $Y_K(r)$ so  $\Delta(s,r) \leq 1$. 
The condition  $\Delta(s,r) \leq 1$ implies that there are at most three of such slopes. Assume $r \in C(s)$ is distinct from $s$, then $\Delta(s,r)=1$ and we have either
 $$ C(s)= \left\lbrace  r, s + r \right\rbrace,\ \ \ \textnormal{or} \ \ \  C(s)= \left\lbrace  r \right\rbrace.$$
Now by homological reason there is a constant $c_K$ independent of $s$ such that
$$|H_1(Y_K(s);\Z)|=c_K \Delta(s,\lambda)$$
and similarly for $r$, see \cite{Watson-PhD} for details.
Therefore since $H_1(Y_K(s);\Z)=H_1(Y_K(r);\Z)=H_1(Y_K(s + r);\Z)$, we must have 
$$\Delta(s,\lambda)=\Delta(r,\lambda)=\Delta(s + r,\lambda).$$
Thus all three elements $s,r,s +r$  lie on  the same line $l$ in $\R^2$ which is at fixed distance from the rational longitude $\lambda$. This is impossible since  $s$ and $r$ are linearly independent. It follows that the first case cannot occur and we have:
$$ C(s)= \left\lbrace r \right\rbrace.$$
Now if $s$ is not a boundary slope  we can apply Lemma \ref{mitovy-norm} to obtain
$$||s||= ||r||.$$ 
Let $k=||s||$, then every element of $C(s)$ lies on the boundary of the ball $B(0,k)$ of the absolute semi-norm. On the other hand if $r\in C(s)$ then $s$ and $r$ must have the same $\mu$ component since they give homeomorphic manifolds. Hence they also lie on a line $l$ parallel to $\lambda$.  From \cite[Proposition 5.2, Proposition 5.3]{Boyer-Zhang} $B(0,k)$ is either a convex polygon or an infinite band. In each cases the line $l$ intersect  $\partial B(0,r)$ twice (at $s$ and $r$) unless $k$ is  a multiple of $s \cdot \lambda$.  This prove  $\sharp C(s)\leq 1$.
\qed 

\bibliographystyle{plain}
\bibliography{PSL2C-character-varieties-and-cosmetic-surgeries}
\thanks{University of Georgia, Department of Mathematics, 606 Boyd GSRC, Athens GA, USA.\\  Email: huygens@cirget.ca}
\end{document}